\documentclass[12pt]{article}
\usepackage{amssymb}
\usepackage{graphicx}
\usepackage[all]{xy}
\usepackage{color}
\usepackage{latexsym}
\usepackage{amssymb,amsmath,amstext}
\usepackage{epsfig}
\usepackage{graphics}
\usepackage{euscript}
\usepackage{ifthen}
\usepackage{wrapfig}
\usepackage{amsthm}
\usepackage{multicol}
\usepackage{paralist}
\usepackage{verbatim}
\usepackage[colorlinks=true, urlcolor=blue, linkcolor=blue, citecolor=blue]{hyperref}
\usepackage[margin=1in]{geometry}
\usepackage{algorithmicx}
\usepackage{algpseudocode}
\usepackage{algorithm}
\usepackage{placeins}
\usepackage{caption}
\usepackage{subcaption}
\usepackage{longtable}
\usepackage{lscape}
\usepackage{authblk}

\numberwithin{equation}{section}
\theoremstyle{plain}%
\newtheorem{theorem}{Theorem}
\numberwithin{theorem}{section}
\newtheorem{proposition}[theorem]{Proposition}
\newtheorem{example}[theorem]{Example}
\newtheorem{lemma}[theorem]{Lemma}

\newtheorem{definition}[theorem]{Definition}
\newtheorem{remark}[theorem]{Remark}

\newcommand{\Z}{\mathbb{Z}}

\allowdisplaybreaks

\begin{document}

\title{Singular Vectors of \\ Orthogonally Decomposable Tensors}
\author{Elina Robeva and Anna Seigal}

\date{ }
\maketitle

\begin{abstract}
Orthogonal decomposition of tensors is a generalization of the singular value decomposition of matrices. In this paper, we study the spectral theory of orthogonally decomposable tensors. For such a tensor, we give a description of its singular vector tuples as a variety in a product of projective spaces.
\end{abstract}

\section{Introduction}
The singular value decomposition of a matrix $M\in\mathbb R^{n_1} \otimes\mathbb R^{n_2}$ expresses it in the form
\begin{align}\label{svd}
M = V^{(1)}\Sigma (V^{(2)})^T = \sum_{i=1}^n \sigma_i v^{(1)}_i \otimes v^{(2)}_i,
\end{align}
where $V^{(1)}\in\mathbb R^{n_1}\otimes \mathbb R^{ n_1}$ and $V^{(2)}\in\mathbb R^{n_2} \otimes\mathbb R^{ n_2}$ are orthogonal matrices. The vectors $v^{(1)}_1,\ldots,v^{(1)}_{n}$ and $v^{(2)}_1,\ldots,v^{(2)}_{n}$ are the columns of the matrices $V^{(1)}$ and $V^{(2)}$ respectively. The matrix $\Sigma$ is diagonal of size $n_1\times n_2$ with non-negative diagonal entries $\sigma_1,...,\sigma_n$, where $n = \min\{n_1,n_2\}$. The singular value decomposition of a matrix is extremely useful for studying matrix-shaped data coming from applications. For example, it allows the best low-rank approximation of a matrix to be found.

In light of the excellent properties of the singular value decomposition, and of the prevalence of tensor data coming from applications, it is a topic of major interest to extend the singular value decomposition to tensors. In fact it is even more crucial to find a low rank approximation of a tensor than it is for a matrix: the greater number of dimensions makes tensors in their original form especially computationally intractable. In this paper we investigate those tensors for which the singular value decomposition is possible. We note that our singular value decomposition is more stringent than that in \cite{LMV}, which is based on flattenings of the tensor.

\begin{definition} A tensor $T\in \mathbb R^{n_1}\otimes \mathbb R^{n_2}\otimes \cdots \otimes \mathbb R^{n_d}$ is {\em orthogonally decomposable}, or {\em odeco}, if it can be written as
$$T = \sum_{i=1}^n \sigma_i v^{(1)}_{i}\otimes v^{(2)}_{i}\otimes \cdots \otimes v^{(d)}_{i},$$
where $n = \min\{n_1,\dots, n_d\}$, the scalars $\sigma_i \in \mathbb R$, and the vectors $v^{(j)}_{1}, v^{(j)}_{2}, \ldots, v^{(j)}_{n} \in \mathbb R^{n_j}$ are orthonormal for every fixed $j\in\{1,\ldots,d\}$.
\end{definition}

We remark that in the above decomposition for $T$ it is sufficient to sum up to $n=\min\{n_1,\dots, n_d\}$ since there are at most $n_j$ orthonormal vectors in $\mathbb R^{n_j}$ for every $j=1,\dots, d$. Such a decomposition will in general be unique up to re-ordering the summands. 

Odeco tensors have been studied in the past due to their appealing properties \cite{AGHKT, AGJ, K03, K01, R, ZG}. Finding the decomposition of a general tensor is NP-hard \cite{HL}, however finding the decomposition of an odeco tensor can be done efficiently via an alternating tensor power method \cite{ZG}.

The variety of odeco tensors was studied in \cite{BDHR}, and the eigenvectors of symmetric odeco tensors of format $n\times\cdots\times n$ were studied in \cite{Robeva}. Here we focus on odeco tensors of format $n_1\times\cdots\times n_d$ that need not be symmetric, and whose dimensions $n_i$ need not be equal. As with matrices, when the dimensions $n_i$ are not equal, it is no longer possible to define eigenvectors. The right notion is now that of a singular vector tuple.


\begin{definition} 
A {\em singular vector tuple} of a tensor $T\in\mathbb R^{n_1}\otimes\cdots\otimes \mathbb R^{n_d}$ is a $d$-tuple of nonzero vectors $(x^{(1)}, \dots, x^{(d)}) \in \mathbb C^{n_1 } \times \cdots \times \mathbb C^{n_d }$ such that
\begin{align}\label{contract}
T( x^{(1)}, \ldots, x^{(j-1)}, \cdot, x^{(j+1)}, \ldots, x^{(d)}) \text{ is parallel to } x^{(j)} \text{, for all} \text{ } j = 1 , \ldots, d
\end{align}
The left hand side of equation \eqref{contract} is the vector obtained by contracting $T$ by the vector $x^{(k)}$ along its $k$-th dimension for all $k\neq j$.
\end{definition}

Since this setup is invariant under scaling each vector $x^{(j)}$, we consider the singular vector tuple $(x^{(1)}, \dots, x^{(d)})$ to lie in the product of projective spaces $\mathbb P^{n_1-1}\times\cdots\times\mathbb P^{n_d-1}$.

The singular vector tuples of a tensor can also be characterized via a variational approach, as in \cite{LHL}. They are the critical points of the optimization problem
\begin{align*}
\text{maximize}& \quad T(x^{(1)}, \dots, x^{(d)})\\
\text{subject to}& \quad ||x^{(1)}|| = \cdots = ||x^{(d)}|| = 1,
\end{align*}
where we note that the global maximizer gives the best rank-one approximation of the tensor.

Given a decomposition of an odeco tensor $T = \sum_{i=1}^n\sigma_i v^{(1)}_i\otimes\cdots \otimes v^{(d)}_{i}$, it is straightforward to see that the tuples $(v^{(1)}_{i}, \ldots, v^{(d)}_{i})$ corresponding to the rank-one tensors in the decomposition are singular vector tuples. For generic matrices $M \in \mathbb{R}^{n_1}\otimes\mathbb R^{n_2}$ the rank-one terms in the singular value decomposition constitute all of the singular vector pairs. In contrast, odeco tensors have additional singular vector tuples that do not appear as terms in the decomposition.

\begin{remark} We distinguish between the cases $T(x^{(1)}, \dots, x^{(d)}) = 0$ and $T(x^{(1)}, \dots, x^{(d)}) \neq 0$. This is equivalent to whether or not the vector in \eqref{contract} is zero. In the former case, the singular vector tuple is a {\em base point} of the maps of projective space $\mathbb P^{n_1-1}\times\cdots\times\mathbb P^{n_{j-1}-1}\times \mathbb P^{n_{j+1}-1}\times\cdots\times\mathbb P^{n_d-1} \to \mathbb P^{n_j-1}$ induced by $T$ for all $j=1,\dots, d$. In the latter case the singular vector tuple is a {\em fixed point} of each of these maps.\end{remark}

Our main theorem is the following description of the complex singular vector tuples of an odeco tensor: 

\begin{theorem}\label{thm:main} 
The projective variety of singular vector tuples of an odeco tensor $T\in \mathbb R^{n_1}\otimes \mathbb R^{n_2}\otimes \cdots \otimes \mathbb R^{n_d}$ consists of
\[ \frac{{(2^{d-1}(d-2) + 1)}^n - 1}{2^{d-1} (d-2)}\]
fixed points, of which $\frac{ {(2^{d-1} + 1)}^n - 1}{2^{d-1}}$ are real, and an arrangement of base points. The base points comprise
${{d \choose 2}}^n - c(d-1)^n + {{c \choose 2}}$
components, each of dimension $\sum_{j = 1}^d ( n_j - 1 ) - 2n$, that are products of linear subspaces of each $\mathbb{P}^{n_j - 1}$. Here, $n = \text{min} \{n_1, \dots, n_d\}$ and $c = \# \{ j : n_j = n \}$.
%
\end{theorem} 

In particular, Theorem~\ref{thm:main} implies that for all but a few small cases the singular vector tuples of an odeco tensor comprise a positive-dimensional variety. In contrast, the variety of singular vector tuples of a generic tensor is zero-dimensional \cite{FO}. It is interesting to study how the positive-dimensional components of the singular vector variety for an odeco tensor adopt generic behavior under a small perturbation. Note the contrast to the variety of eigenvectors of a symmetric odeco tensor, which is also zero-dimensional \cite{Robeva}.
%

The rest of this paper is organized as follows. In Section \ref{sec:eigenvectors}, we use the theory of binomial ideals \cite{ES} to describe the singular vector tuples of an odeco tensor. In Section \ref{sec:proof} we conclude the proof of our theorem by describing the positive-dimensional components of the variety of singular vector tuples. Finally, in Section 4, we explore the structure of these components in more detail by studying specific examples.


\section{Description of the Singular Vector Tuples}\label{sec:eigenvectors}

In this section we give a formula for the singular vector tuples of an odeco tensor. We start by considering a diagonal odeco tensor.

\begin{lemma}\label{lem:singular} Let $S\in\mathbb R^{n_1}\otimes \cdots \otimes \mathbb R^{n_d}$ be the tensor
$$S = \sum_{i=1}^n\sigma_i e^{(1)}_i\otimes \cdots \otimes e_i^{(d)},$$
where $\sigma_1, \ldots ,\sigma_n\neq0$, the vector $e^{(j)}_i$ is the $i$th basis vector in $\mathbb{R}^{n_j}$, and $n = \min\{n_1,\dots, n_d\}$.
The singular vector tuples $( x^{(1)}, \ldots , x^{(d)}) \in\mathbb P^{n_1-1}\times \ldots \times\mathbb P^{n_d-1}$ of $S$ are given as follows.
\begin{enumerate}

\item[\underline{Type I:}]

Tuples $( x^{(1)}, \ldots , x^{(d)})$ of the form 
\begin{align}\label{eqn:diagodeco}  \sigma_{\tau(1)}^{-\frac1{d-2}} \begin{pmatrix}e^{(1)}_{\tau(1)}, e^{(2)}_{\tau(1)},\dots, e^{(d)}_{\tau(1)}\end{pmatrix} + \sum_{i = 1}^m \eta_i \sigma_{\tau(i)}^{-\frac1{d-2}}  \begin{pmatrix} e^{(1)}_{\tau(i)}, \chi_{i}^{(2)}e^{(2)}_{\tau(i)},\dots, \chi_{i}^{(d)}e^{(d)}_{\tau(i)}\end{pmatrix} \end{align}

where $1\leq m \leq n$, the scalars $\chi^{(j)}_{i} \in\{\pm 1\}$ are such that $\prod_{j=2}^d \chi^{(j)}_{i} = 1$ for every $i=1,...,m$, each scalar $\eta_{i}$ is a $(2d-4)$-th root of unity, and $\tau$ is any permutation on $\{1,\dots, n\}$. The tuples with real coordinates are those for which each scalar $\eta_i$ is a real root of unity, i.e. $\eta_i \in \{ \pm 1\}$.

\item[\underline{Type II:}]

All tuples $( x^{(1)}, \ldots , x^{(d)})$ such that the $n \times d$ matrix $X = (x^{(j)}_{i})_{1\leq i \leq n, 1 \leq j\leq d}$ has at least two zeros in each row. Since each $x^{(j)} \in\mathbb P^{n_j-1}$, we further require that no $x^{(j)}$ has all coordinates equal to zero.
\end{enumerate}
\end{lemma}

 Before proving Lemma \ref{lem:singular}, we illustrate it by way of the following example:

\begin{example}\label{ex233} Consider the odeco tensor $S = e_1^{(1)} \otimes e_1^{(2)} \otimes e_1^{(3)} + e_2^{(1)} \otimes e_2^{(2)} \otimes e_2^{(3)} \in\mathbb R^2\otimes \mathbb R^3\otimes \mathbb R^3$. Its Type I singular vector tuples are

$$\begin{pmatrix}e_1^{(1)}, e_1^{(2)}, e_1^{(3)} \end{pmatrix},
\begin{pmatrix}e_2^{(1)}, e_2^{(2)}, e_2^{(3)} \end{pmatrix} $$
$$
\begin{pmatrix}e_1^{(1)} + e_2^{(1)}, e_1^{(2)} + e_2^{(2)}, e_1^{(3)} + e_2^{(3)} \end{pmatrix},
\begin{pmatrix}e_1^{(1)} + e_2^{(1)}, e_1^{(2)} - e_2^{(2)}, e_1^{(3)} - e_2^{(3)} \end{pmatrix}, $$
$$\begin{pmatrix}e_1^{(1)} - e_2^{(1)}, e_1^{(2)} + e_2^{(2)}, e_1^{(3)} - e_2^{(3)} \end{pmatrix},
\begin{pmatrix}e_1^{(1)} - e_2^{(1)}, e_1^{(2)} - e_2^{(2)}, e_1^{(3)} + e_2^{(3)} \end{pmatrix},
$$

%
%
%
The Type II singular vectors make five copies of $\mathbb P^1$, namely

$$\begin{pmatrix} \square e_1^{(1)} + \square e_2^{(1)}, e_3^{(2)}, e_3^{(3)}
\end{pmatrix},
\begin{pmatrix}e_1^{(1)}, \square e_2^{(2)} + \square e_3^{(2)}, e_3^{(3)}\end{pmatrix}, 
\begin{pmatrix}e_1^{(1)}, e_3^{(2)},  \square e_2^{(3)} + \square e_3^{(3)}\end{pmatrix}, $$
$$\begin{pmatrix}e_1^{(1)}, \square e_1^{(2)} + \square e_3^{(2)}, e_3^{(3)}\end{pmatrix}, 
\begin{pmatrix}e_2^{(1)}, e_3^{(2)},  \square e_1^{(3)} + \square e_3^{(3)}\end{pmatrix},$$
where two $\square$'s in a vector indicate a copy of $\mathbb{P}^1$ on those two coordinates. The five copies of $\mathbb{P}^1$ intersect in two triple intersections, as seen in Figure \ref{fig:233a1}.

\begin{figure}[h]
\centering
\includegraphics[width = 0.4 \linewidth]{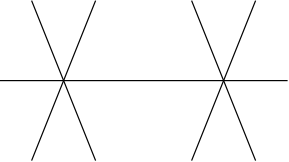}
\caption{The Type II singular vectors: five copies of $\mathbb{P}^1$ meeting at two triple intersections}
\label{fig:233a1}
\end{figure}
According to \cite{FO}, the generic number of singular vector tuples of a tensor of this size is 15, so the five copies of $\mathbb P^1$ degenerate from nine points. For example, consider the family of perturbed tensors 
$$S_\epsilon = S + \epsilon T,$$
where $T$ is the $2\times 3\times 3$ tensor with slices $T_{1,\cdot,\cdot}$ and $T_{2, \cdot, \cdot}$ given by
$$T_{1,\cdot, \cdot} = \begin{pmatrix} 0 & 40 & 10\\
100&  3& 3\\
3& 2 & 6
\end{pmatrix}, \quad\quad
T_{2, \cdot, \cdot} = \begin{pmatrix}
7 & 1 &1 \\
8& 0 &2\\
2& 2 &3
\end{pmatrix}.$$
For $\epsilon$ on the order of $10^{-6}$ we attain nine points: one point near each copy of $\mathbb{P}^1$, and two points of multiplicity 2 near each triple intersection.
\end{example}
We will return this example in Section \ref{sec:4}.

\begin{proof}[Proof of Lemma \ref{lem:singular}]
By definition, $(x^{(1)}, \ldots , x^{(d)})$ is a singular vector tuple of $S$ if and only if for each $j=1,\dots, d$ the following matrix has rank at most one:
$$M_{S, j} = \begin{bmatrix} S(x^{(1)}, \ldots, x^{(j-1)}, \cdot, x^{(j+1)}, \ldots, x^{(d)} ) &| & x^{(j)}  
\end{bmatrix}
=\begin{bmatrix}\sigma_1 x^{(1)}_{1}\cdots \hat{x}^{(j)}_{1}\cdots x^{(d)}_{1} & x^{(j)}_{1}\\
\vdots & \vdots\\
\sigma_n x^{(1)}_{n}\cdots \hat{x}^{(j)}_{n}\cdots x^{(d)}_{n} & x^{(j)}_{n}
\end{bmatrix}$$
where $\hat{x}^{(j)}_i$ denotes the omission of $x^{(j)}_i$ from the product.

We examine the structure of the singular vectors tuples of $S$ by looking at the following three cases.

\bigskip
\underline{\textbf{Case 1:}} Consider the variables $x^{(1)}_{i}, \ldots, x^{(d)}_{i}$, where $i\in \{1,2,\ldots,n\}$ is fixed. Suppose that exactly one of the variables $x^{(j)}_{i} = 0$, i.e. $x^{(k)}_{i}\neq 0$ for all $k\neq j$. The $i$-th row of the matrix $M_{S, j}$ has first entry $\sigma_i x^{(1)}_{i}\ldots \hat{x}^{(j)}_{i} \ldots x^{(d)}_{i} \neq 0$ and second entry $x^{(j)}_{i} = 0$. Therefore, in order for this matrix to have rank $1$, we need the whole second column to be zero, i.e. $x^{(j)}_{1} = \cdots = x^{(j)}_{n} = 0$. Since $x^{(j)} \in\mathbb P^{n_j-1}$, this can only happen if $n_j > n$ and one of the last $n_j - n$ coordinates of $x^{(j)}$ is nonzero. But the contraction $S( x^{(1)}, \ldots , x^{(j-1)}, \cdot, x^{(j+1)}, \ldots x^{(d)})$ lies in the span of $e^{(j)}_1,\dots, e^{(j)}_n$, so in order for it to be parallel to $x^{(j)}$ it has to be 0. In particular, its $i$-th entry $\sigma_i x^{(1)}_{i}\ldots\hat{x}^{(j)}_{i}\ldots x^{(d)}_{i}$ has to be 0. Contradiction!  Therefore, we can't have exactly one of the variables $x^{(1)}_{i}, \ldots, x^{(d)}_{i}$ equal to $0$.

\bigskip
\underline{\textbf{Case 2:}} Suppose that for some $i$ at least two of the entries $x^{(1)}_{i}, \ldots, x^{(d)}_{i}$, but not all of them, are equal to $0$. This means that the the entry in the $i$-th row and the first column of $M_{S, k}$ is $0$ for every $k$ and if $x^{(k)}_{i}\neq 0$ (and we assumed that one such $k$ exists), then, the entry in the $i$-th row and the second column is not $0$. For such a $k$, the whole first column of $M_{S, k}$ must be $0$ in order that it have rank 1. Therefore, for every $i$, at least two of the entries $x^{(1)}_{i}, \ldots, x^{(d)}_{i}$ are equal to $0$. Conversely, if for every $i$ at least two of the entries $x^{(1)}_{i}, \ldots, x^{(d)}_{i}$ are equal to $0$ in such a way that $x^{(j)} \in \mathbb P^{n_j-1}$, then, $( x^{(1)}, \ldots,  x^{(d)} )$ is a singular vector tuple of $S$. This gives the singular vector tuples of Type II, also known as the base points.

\bigskip
\underline{\textbf{Case 3:}} It remains to consider the situation where, for every $i$, either $x^{(1)}_{i} = \ldots = x^{(d)}_{i} = 0$ or none of the variables $x^{(1)}_{i}, \ldots , x^{(d)}_{i}$ are $0$. After reordering, assume that $x^{(1)}_i = \ldots = x^{(d)}_i = 0$ for $m+1 \leq i \leq n$, for some $m\leq n$, and $x^{(1)}_i, \ldots, x^{(d)}_i \neq 0$ for $1 \leq i \leq m$.

The condition for being a singular vector tuple now yields the following system of polynomial equations in the Laurent polynomial ring $\mathbb{C} \left[ x^{(j)}_i, \frac{1}{x^{(j)}_i} : 1 \leq i \leq m, 1 \leq j \leq d \right]$:

\begin{align}\label{equations}
I = \left\langle \sigma_i x^{(1)}_{i}\ldots\hat{x}^{(j)}_{i}\ldots x^{(d)}_{i} x^{(j)}_{l} = \sigma_l x^{(1)}_{l}\ldots\hat{x}^{(j)}_{l} \ldots x^{(d)}_{l}x^{(j)}_{i} : 1\leq j \leq d, 1 \leq i,l \leq m \right\rangle
\end{align}
To solve this system of equations, we use the theory of binomial ideal decomposition developed in \cite{ES}.

Consider the lattice
$$ L_\rho = \left\langle \sum_{k = 1}^d (e^{(k)}_i - e^{(k)}_l ) - 2 ( e^{(j)}_i - e^{(j)}_l ) : 1 \leq j \leq d, 1 \leq i, l \leq m \right\rangle \subseteq \mathbb{Z}^{d \times m} $$
where $e^{(a)}_b$ is the elementary basis vector in $\mathbb{Z}^{d \times m}$ with a 1 in coordinate $(a,b)$. Let $\rho : L_\rho \to \mathbb{C}^*$ denote the partial character
\begin{align}\label{partialchar} \rho \left( \sum_{k = 1}^d ( e^{(k)}_i - e^{(k)}_l ) - 2 ( e^{(j)}_i - e^{(j)}_l ) \right) = \frac{\sigma_l}{\sigma_i} \quad \forall \text{ }1 \leq j \leq d, 1 \leq i, l, \leq m \end{align}
Then the lattice ideal $I(\rho) = \langle x^v - \rho(v) : v \in L_\rho \rangle $ is our ideal $I$, where $x^v$ denotes taking the variables $x^{(j)}_i$ in the ring to the powers indicated by the lattice element $v$.

We have the inclusion $L_\rho \subseteq L = \langle e^{(j)}_i - e^{(j)}_l : 1 \leq j \leq d, 1 \leq i,l \leq m \rangle$. Therefore by \cite[Theorem 2.1]{ES},
$$I(\rho) = \bigcap_{\rho' \text{ extends }\rho \text{ to } L} I(\rho').$$

To decompose the ideal $I(\rho)$, we therefore seek to characterize the partial characters $\rho'$ of $L$ which extend $\rho$. Summing \eqref{partialchar} over $1 \leq j \leq d$ gives the formula
$$\rho\left(\sum_{j=1}^d \left( \sum_{k=1}^d (e^{(k)}_{i} - e^{(k)}_l ) - 2 ( e^{(j)}_i - e^{(j)}_{l} ) \right) \right) = \left( \frac{\sigma_l}{\sigma_i}\right)^d$$
which, after simplifying, yields $\rho \left((d-2)\sum_{k = 1}^d (e^{(k)}_{i}-e^{(k)}_{l})\right) = \left(\frac{\sigma_l}{\sigma_i}\right)^d$. Therefore, any $\rho'$ extending $\rho$ satisfies
\begin{align}\label{eq1}
\rho' \left( \sum_{k = 1}^d (e^{(k)}_i - e^{(k)}_l ) \right) = \phi_{il} {\left( \frac{\sigma_l}{\sigma_i} \right) }^{ \frac{d}{d-2}}
\end{align}
where $\phi_{il}$ is a $(d-2)$-th root of unity. By rearranging \eqref{partialchar}, we furthermore see that any such $\rho'$ must satisfy
\begin{align}\label{eq2}
\rho' \left( 2 ( e^{(j)}_i - e^{(j)}_l )\right) = \rho' \left( \sum_{k = 1}^d (e^{(k)}_i - e^{(k)}_l) \right) \left( \frac{\sigma_i}{\sigma_l} \right).
\end{align}
Combining \eqref{eq1} and \eqref{eq2} yields
$$\rho' \left( 2 ( e^{(j)}_i - e^{(j)}_l )\right) = \phi_{il}\left(\frac{\sigma_l}{\sigma_i}\right)^{\frac2{d-2}}.$$
Thus,
$$ \rho' \left( e^{(j)}_i - e^{(j)}_l \right) =  \phi^{(j)}_{il} \left( \frac{\sigma_l}{\sigma_i} \right)^{\frac{1}{d-2}} $$
where $\phi_{il}^{(j)}$ are $2(d-2)$-th roots of unity such that $(\phi_{il}^{(j)})^2 = \phi_{il}$ for all $j=1,\dots, d$. It remains to find the relations satisfied by the $\phi_{il}^{(j)}$ as $i, l, j$ vary so that the original equation~\eqref{partialchar} is satisfied. Substituting our expression for $\rho'$ in to \eqref{partialchar} yields
$$\prod_{k=1}^d\left(\phi_{il}^{(k)} \left( \frac{\sigma_l}{\sigma_i}\right)^{\frac1{d-2}}\right) \phi_{il}^{-1}\left(\frac{\sigma_l}{\sigma_i}\right)^{-\frac2{d-2}} = \frac{\sigma_l}{\sigma_i},$$
which is equivalent to
\begin{align}\label{eq:necessaryEquation}
\prod_{k=1}^d\phi_{il}^{(k)} = \phi_{il}.
\end{align}


To satisfy these conditions, we express the roots of unity in the following way. Denote the $(2d-4)$-th root of unity $\phi_{il}^{(1)}$ by $\eta_{il}$. Since $\eta_{il}^2 = \phi_{il} = (\phi_{il}^{(j)})^2$ for all $j = 1, \ldots, d$, each $\phi_{il}^{(j)}$ can be written in terms of $\eta_{il}$ as $\phi_{il}^{(j)} = \eta_{il}\chi_{il}^{(j)}$, where $\chi_{il}^{(j)} \in \{ \pm 1 \}$. Equation~\eqref{eq:necessaryEquation} can now be written as
$$\eta_{il}^d \prod_{j=2}^d\chi_{il}^{(j)} = \phi_{il} = \eta_{il}^2.$$
Equivalently,
$$\eta_{il}^{d-2} = \prod_{j=2}^d \chi_{il}^{(j)},$$
where we note that since $\eta_{il}$ is a $(2d-4)$-th root of unity, the multiple $\eta_{il}^{d-2} \in \{ \pm 1 \}$.

Finally, since $(e^{(j)}_i - e^{(j)}_l ) + (e^{(j)}_l - e^{(j)}_h ) + (e^{(j)}_h - e^{(j)}_i ) = 0$, applying $\rho$ gives 
$$1 = \chi^{(j)}_{il} \eta_{il} {\left( \frac{\sigma_l}{\sigma_i} \right)}^{\frac{1}{d-2}} \chi^{(j)}_{lh} \eta_{lh} {\left( \frac{\sigma_h}{\sigma_l} \right)}^{\frac{1}{d-2}} \chi^{(j)}_{hi} \eta_{hi} {\left( \frac{\sigma_i}{\sigma_h} \right)}^{\frac{1}{d-2}}.$$

We now have all the relations required to find the ideals $I(\rho')$:
$$ I(\rho') = \left\langle x^{(j)}_i - \chi^{(j)}_{i1} \eta_{i1} {\left( \frac{\sigma_1}{\sigma_i} \right)}^{\frac{1}{d-2}} x^{(j)}_1 : 1 \leq i \leq m, 1 \leq j \leq d \right \rangle $$
where $\chi^{(j)}_{i1} \in \{ \pm 1 \}$ are such that $\chi^{(1)}_{i1} = 1$, $\prod_{j = 2}^d \chi^{(j)}_{i1} = 1$ and the $\eta_{i1}$ are $(2d-4)$-th roots of unity. Setting $\chi^{(j)}_{i} = \chi^{(j)}_{i1}$ and $\eta_i = \eta_{i1}$, and taking $I$ to be the intersection of the $I(\rho')$, we obtain the required form of our singular vector tuples:
$$ I = \bigcap_{\eta, \chi} \left \langle x^{(j)}_i - \chi^{(j)}_i \eta_i { \left( \frac{\sigma_1}{\sigma_i} \right) }^{\frac{1}{d-2} } x_{1}^{(j)}: 1 \leq i \leq m, 1 \leq j \leq d \right \rangle .$$
These are the singular vector tuples of Type 1, also known as the fixed points.
\end{proof}

Now, we proceed to the main result of this section. We describe the singular vector tuples of a general odeco tensor.

\begin{proposition}\label{thm:general_tensor} Let $T = \sum_{i=1}^n \sigma_iv^{(1)}_{i}\otimes\cdots\otimes v^{(d)}_{i}\in\mathbb R^{n_1}\otimes\cdots\otimes \mathbb R^{n_d}$ be an odeco tensor such that $v^{(j)}_{1}, \ldots, v^{(j)}_{n}\in\mathbb R^{n_j}$ are orthonormal vectors. Let $V^{(j)}\in\mathbb R^{n_j}\otimes \mathbb R^{n_j}$ be any orthogonal matrix whose first $n$ columns are $v^{(j)}_{1}, \ldots , v^{(j)}_{n}$. Then, the singular vector tuples of $T$ are given by $( V^{(1)}x^{(1)},  \ldots , V^{(d)}x^{(d)})$ where $( x^{(1)}, \ldots , x^{d)})$ is a singular vector tuple of the diagonal tensor $S = \sum_{i=1}^n\sigma_i e^{(1)}_{i}\otimes\cdots\otimes e^{(d)}_i$ described in Lemma \ref{lem:singular}. In other words, the singular vectors of $T$ are as follows.
\begin{enumerate}
\item[\underline{Type I:}]
Tuples $\begin{pmatrix} V^{(1)}x^{(1)}, \dots, V^{(d)} x^{(d)}\end{pmatrix},$
such that $(x^{(1)}, \ldots , x^{d)})$ is a Type I singular vector of the diagonal odeco tensor in Lemma \ref{eqn:diagodeco}.

\item[\underline{Type II:}] Tuples $\begin{pmatrix} V^{(1)}x^{(1)}, \dots, V^{(d)} x^{(d)}\end{pmatrix},$ where the matrix $X = (x^{(j)}_{i})_{ij}$ has at least two zeros in each row such that none of the vectors $x^{(j)} \in\mathbb P^{n_j-1}$ is identically zero.
\end{enumerate}
\end{proposition}

\begin{proof}
Assume that $(y^{(1)}, \ldots, y^{(d)})$ is a singular vector tuple of $T$. Equivalently, for all $1 \leq j \leq d$, the vector $T(y^{(1)}, \ldots, y^{(j-1)}, \cdot, y^{(j+1)}, \ldots, y^{(d)})$ is parallel to $y^{(j)}$. Unpacking the definition of the contraction, we obtain
\begin{align}\label{2.2eq} T(y^{(1)}, \ldots, y^{(j-1)}, \cdot, y^{(j+1)}, \ldots, y^{(d)}) = \sum_{i = 1}^n \sigma_i \left( \prod_{k \neq j} ( v_i^{(k)} \cdot y^{(k)} ) \right) v_i^{(j)} \end{align}
The inner-product term $(v_i^{(k)} \cdot y^{(k)})$ is the $i$-th element in the vector $x^{(k)} := {(V^{(k)})}^T y^{(k)}$, where $V^{(k)}$ is any orthogonal matrix with first $n$ columns equal to $v_1^{(k)}, \ldots, v_n^{(k)}$. We can re-write the right hand side of \eqref{2.2eq} in terms of the $x^{(k)}$, $1 \leq k \leq d$, as
$$ \sum_{i = 1}^n \sigma_i \left( \prod_{k \neq j} x^{(k)}_i \right) v^{(j)}_i = V^{(j)} \left( \sum_{i = 1}^n \sigma_i \left( \prod_{k \neq j} x^{(k)}_i \right) e_i^{(j)} \right).$$
Therefore,
$$ T(y^{(1)}, \ldots, y^{(j-1)}, \cdot, y^{(j+1)}, \ldots, y^{(d)}) = V^{(j)} S (x^{(1)}, \ldots, x^{(j-1)}, \cdot, x^{(j+1)}, \ldots, x^{(d)} ) ,$$
where $S = \sum_{i = 1}^n \sigma_i e^{(1)}_i \otimes \cdots \otimes e^{(d)}_i$. Since $V^{(j)}$ is orthogonal, $T(y^{(1)}, \dots, y^{(j-1)}, \cdot, y^{(j+1)}, \ldots, y^{(d)})$ and $y^{(j)} = V^{(j)}x^{(j)}$ are parallel if and only if $S(x^{(1)}, \dots, x^{(j-1)}, \cdot, x^{(j+1)}, \ldots, x^{(d)})$ and $x^{(j)}$ are parallel. Therefore, equivalently $(x^{(1)}, \dots, x^{(d)})$ is a singular vector tuple of $S$, and the solutions for all such $(x^{(1)}, \dots, x^{(d)})$ are given in Lemma~\ref{lem:singular}.
\end{proof}

\section{Proof of the Main Theorem}\label{sec:proof}

\begin{proof}[Proof of Theorem \ref{thm:main}]
The count for the contribution of the fixed points to the projective variety of singular vector tuples is obtained as follows directly from Proposition \ref{thm:general_tensor}. For any choice of $m\in \{1, \dots, n\}$, a subset of $\{1,\dots, n\}$ of size $m$, scalars $\eta_i$ which are $(2d-4)$-th roots of unity (where $i\in\{2,\dots,m\}$), and $\chi_i^{(j)}\in\{\pm 1\}$ such that $\prod_{j=2}^d \chi_i^{(j)} = 1$ (where $i\in\{2,\dots,m\}$ and $j\in\{2, \dots, d\}$), we have one singular vector tuple. Therefore, the total number of singular vector tuples of Type I is
$$\sum_{m = 1}^n \binom n m (2d-4)^{m-1} 2^{(m-1)(d-2)} = \frac{(2^{d-1}(d-2) + 1)^n - 1}{2^{d-1}(d-2)}.$$
If we insist that the singular vector tuples be real, we have only two values for the choice of each $\eta_i$ rather than $(2d-4)$ which yields the real count of $\frac{ {(2^{d-1} + 1)}^n - 1}{2^{d-1}}$.

It remains to study the contribution made by the Type II singular vector tuples which constitute the base locus. By Proposition \ref{thm:general_tensor}, we can restrict our attention to the tensor $S = \sum_{i=1}^n\sigma_i e_i\otimes\cdots\otimes e_i$, since its singular vector tuples differ from those of a general tensor only by an orthogonal change of coordinates in each factor.

We first study the case in which all dimensions are equal, $n_1 = \cdots = n_d = n$. Here, the tuple $(x^{(1)}, \dots, x^{(d)})$ is a Type II singular vector tuple if and only if the matrix $X = (x^{(j)}_i)$ has at least two zeros in every row and none of the vectors $x^{(j)}$ is identically zero. This configuration is a subvariety of $\mathbb P^{n-1}\times\cdots\times \mathbb P^{n-1}$. Its ideal is given by
\begin{align}\label{idealeq}
\sum_{i=1}^n\langle x^{(1)}_{i}\cdots \hat{x}^{(j)}_{i}\cdots x^{(d)}_{i}: j=1,\dots, d\rangle \quad = \quad\sum_{i=1}^n \bigcap_{1\leq j < k \leq d} \langle x^{(j)}_{i}, x^{(k)}_{i}\rangle.
\end{align}

We count the number of components in this subvariety by looking at the Chow ring of $\mathbb P^{n-1}\times\cdots\times \mathbb P^{n-1}$, which is $\mathbb Z[t_1,\dots, t_d]/ {\langle t_1^n, \dots, t_d^n \rangle}$. Each $t_j$ represents the class of a hyperplane in $\mathbb{P}^{n_j - 1}$, the $j$th projective space in the product. The equivalence class of the variety $\mathcal V\left(\langle x^{(j)}_{i}, x^{(k)}_{i}\rangle\right)$ is given by $t_j t_k$. We consider the variety
\begin{align}\mathcal V\left(\bigcap_{1\leq j < k \leq d} \langle x^{(j)}_{i}, x^{(k)}_{i}\rangle\right) =\bigcup_{1\leq j < k \leq d} \mathcal V\left(\langle x^{(j)}_{i}, x^{(k)}_{i}\rangle\right)
\end{align}
which yield our variety of interest when we intersect over $i$. Its equivalence class is given by $\sum_{1\leq j < k\leq d} t_jt_k$. From this, we see that the equivalence class in the Chow ring of the total configuration is given by
\begin{align}\label{polyp}
p(t_1,\dots, t_d) = \left(\sum_{1\leq j < k\leq d} t_j t_k\right)^n.
\end{align}

Therefore, to count the number of linear spaces that constitute the Type II singular vector tuples, we wish to count the number of monomials of the polynomial \eqref{polyp} as an element of the Chow ring. Equivalently we count the terms in the expansion, as an element of $\mathbb Z[t_1,\dots, t_d]$, that are not divisible by $t_j^d$ for any $j$.


A monomial in the expanded form of \eqref{polyp} is produced by multiplying one of the ${{d \choose 2}}$ terms in each of the $n$ factors. This produces the first term, ${{d \choose 2}}^n$, in the expression for the number of components in the base locus. We must now subtract those terms that are divisible by $t_j^d$ for some fixed $j$. These are formed by selecting the terms $t_j t_{k_1}, \ldots, t_j t_{k_n}$ from consecutive factors. There are $d-1$ choices for each $k_s$, and $d$ choices for the fixed $j$, yielding at first glance $d(d-1)^n$ terms of this format. However, we have double-counted those terms of the form $t_j^n t_k^n$ for fixed $j$ and $k$, of which there are ${{d \choose 2}}$. Combining these terms gives the correct specialization of our desired formula to the case $c = \# \{ j : n_j = n \} = d$:
\begin{align}\label{eqdim}
{{d \choose 2}}^n - d(d-1)^n + {{d \choose 2}} 
\end{align}
The codimension of the ideal in \eqref{idealeq} is $2n$, so our linear spaces enumerated above are of dimension $d(n-1) - 2n$.

The case of non-equal dimensions follows similarly: consider $S = \sum_{i=1}^n\sigma_i e^{(1)}_i\otimes\cdots\otimes e^{(d)}_i$ of format $n_1 \times \cdots \times n_d$ where $n = \min \{ n_1, \ldots, n_d \}$ and $c = \# \{ j : n_j = n \}$. To count the number of maximal-dimensional linear spaces, we consider the same polynomial \eqref{polyp} in the Chow ring $\Z[t_1, \ldots, t_d]/{\langle t_1^{n_1}, \dots, t_d^{n_d} \rangle }$, and we now want to count the number of terms which are not divisible by $t_j^{n_j}$ for any $j = 1, \ldots, d$.

From the form of $p$ in \eqref{polyp}, we see that it is impossible for a term to be divisible by $t_j^{n_j}$ for any $n_j > n$. Our previous formula \eqref{eqdim} therefore generalizes to
\[ {{d \choose 2}}^n - c(d-1)^n + {{c \choose 2}} \]
and the dimension of each components is $\sum_{j = 1}^d ( n_j - 1 ) - 2n$. This concludes the proof.
\end{proof}


\section{Further Explorations of the Type II Singular Vectors}\label{sec:4}

In this section we turn our attention to the Type II singular vector tuples of the odeco tensor $S = \sum_{i=1}^n e^{(1)}_i\otimes\cdots\otimes e^{(d)}_i$, where $S$ is of format $n_1 \times \cdots \times n_d$ and $n = \min\{ n_1, \ldots, n_d \}$.

We can associate to each projective space $\mathbb{P}^{n_j - 1}$ the simplex $\Delta_{n_j-1}$ and consider our linear spaces as polyhedral subcomplexes (prodsimplicial complexes) in the boundary of the product of simplices $\Delta_{n_1 - 1} \times \ldots \times \Delta_{n_d - 1}$. The number of components in the variety of Type~II singular vector tuples is the number of facets in this complex.

We first return to Example \ref{ex233}, in which we had six Type I singular vector tuples, and the Type II singular vector tuples made up five copies of $\mathbb{P}^1$. In Figure \ref{fig:233b}, we draw the polyhedral complex in $\Delta_1 \times \Delta_2 \times \Delta_2$ corresponding to the Type II singular vector tuples. Motivated by this example, we investigate the shape of the Type II singular vector tuples of other small odeco tensors.

\begin{figure}[h]
\centering
  \includegraphics[width=.2\linewidth]{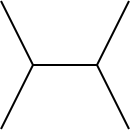}
\caption{The Type II singular vectors tuples of a $2 \times 3 \times 3$ odeco tensor, drawn as a polyhedral complex}
\label{fig:233b}
\end{figure}

It is interesting to stratify odeco tensors according to the dimension of their Type II singular vectors, using the following proposition:

\begin{proposition}\label{thm:finitek} 
For each dimension $k$, the odeco tensors whose Type II singular vector tuples have dimension $k$ come from a finite list of possible sizes $n_1 \times \cdots \times n_d$.
\end{proposition}

\begin{proof}
By Theorem \ref{thm:main}, we seek the solutions of $n_1, \ldots, n_d$ with $n_j \geq 2$ and $d \geq 3$ to the equation
\begin{align}\label{solv}
\sum_{j = 1}^d (n_j - 1) - 2n = k
\end{align}
where $n = \text{min} \{ n_1, \ldots, n_d \}$. An odeco tensor of size $n_1 \times \cdots \times n_d$ will then have Type II singular vector tuples consisting of product of linear spaces of dimension $k$. Without loss of generality, we assume that $n_1 \leq \ldots \leq n_d$, and hence $n = n_1$. Let the constant $\alpha$ be such that $n_2 = n + \alpha$. For fixed $\alpha$, rearranging \eqref{solv} shows that we seek to solve the equation
\begin{align}
\sum_{j = 3}^d (n_j - 1) = k + 2 - \alpha .
\end{align}
This has finitely many solutions, since the right hand side is a fixed number, and each summand on the left hand side has strictly positive integer size. From the form of the right hand side, we see that there will be solutions for only finitely many values of $\alpha$. In conclusion, there are only finitely many size combinations $n_1 \times \cdots \times n_d$ which yield Type II singular vector tuples of dimension $k$.
\end{proof}

For example, odeco tensors whose Type II singular vector tuples constitute a zero-dimensional projective variety have possible sizes:
$$ \{ 2 \times 2 \times 2, 3 \times 3 \times 3 , 2 \times 2 \times 2 \times 2 \} .$$

Theorem \ref{thm:main} tells us how many singular vector tuples there are of Types I and II, which are entered in the first two columns of the table below. The number of singular vector tuples of a generic tensor of a given format is given by \cite[Theorem 1]{FO}, and this is entered into the last column of the table. We observe that odeco tensors whose Type II singular vector tuples consist solely of points attain the generic count.
 
\begin{center}
\begin{tabular}{| l | l | l | l | }
\hline
Tensor Size & Type I Count & Type II Count & Generic Count \\ \hline
$2 \times 2 \times 2$ & 6 & 0 & 6 \\ \hline
$3 \times 3 \times 3$ & 31 & 6 & 37 \\ \hline
$2 \times 2 \times 2 \times 2$ & 18 & 6 & 24 \\ \hline
\end{tabular}
\end{center}

Now we consider odeco tensors whose Type II singular vector tuples make a one-dimensional projective variety. They are of one of the following formats:
$$ \{ 2 \times 3 \times 3 , 2 \times 2 \times 4, 3 \times 3 \times 4, 4 \times 4 \times 4, 2 \times 2 \times 2 \times 3, 2 \times 2 \times 2 \times 2 \times 2 \}.$$
Their singular vector tuples consists of a finite collection of points (Type~I) and a collection of copies of $\mathbb{P}^1$ in the product of projective spaces $\mathbb{P}^{n_1 - 1} \times \ldots \times \mathbb{P}^{n_d - 1}$ (Type~II). When two copies of $\mathbb{P}^1$ meet, they do so at a triple intersection point. The data for these tensor formats is recorded in the table below. Under a small perturbation, each copy of $\mathbb{P}^1$ contributes one singular vector tuple, and two arise from each triple intersection. We observe that summing the Type I count, the number of copies of $\mathbb{P}^1$, and twice the number of triple intersections yields the generic count.

\begin{center}
\begin{tabular}{| l | l | l | l | l | }
\hline
Tensor Size & Type I Count & \#$\mathbb{P}^1$s & \#Triple Intersections & Generic Count \\ \hline
$2 \times 3 \times 3$ & 6 & 5 & 2 & 15 \\ \hline
$2 \times 2 \times 4$ & 6 & 2 & 0 & 8  \\ \hline
$3 \times 3 \times 4$ & 31 & 12 & 6 & 55 \\ \hline
$4 \times 4 \times 4$ & 156 & 36 & 24 & 240 \\ \hline
$2 \times 2 \times 2 \times 3$ & 18 & 12 & 6 & 42 \\ \hline
$2 \times 2 \times 2 \times 2 \times 2 $ & 50 & 30 & 20 & 120 \\ \hline
\end{tabular}
\end{center}

We explored the $2 \times 3 \times 3$ case in more detail in Example \ref{ex233}. In the $3 \times 3 \times 4$ and $2 \times 2 \times 2 \times 3$ cases the simplicial complexes of the Type II singular vector tuples are the same shape. They consist of the 12 copies of $\mathbb{P}^1$ meeting at six triple intersections pictured in Figure~\ref{fig:334}.

\begin{figure}[h]
\centering
   \includegraphics[width=6cm]{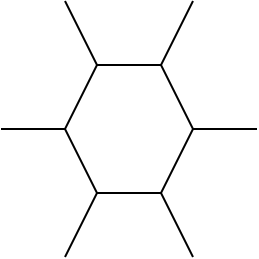}
  \caption{The 12 copies of $\mathbb{P}^1$ with six triple intersection points, for $3 \times 3 \times 4$ tensors and $2 \times 2 \times 2 \times 3$ tensors}
\label{fig:334}
\end{figure}

In the case of $2 \times 2 \times 2 \times 2 \times 2$ odeco tensors, we have 30 copies of $\mathbb{P}^1$ that meet at 20 triple intersection points as seen in the non-planar arrangement pictured in Figure~\ref{fig:22222}. In the case of $4 \times 4 \times 4$ odeco tensors, we have 36 copies of $\mathbb{P}^1$ meeting at 24 triple intersection points as pictured in Figure \ref{fig:444}.

\begin{figure}
\centering  \vspace{-2ex}
   \includegraphics[width=7cm]{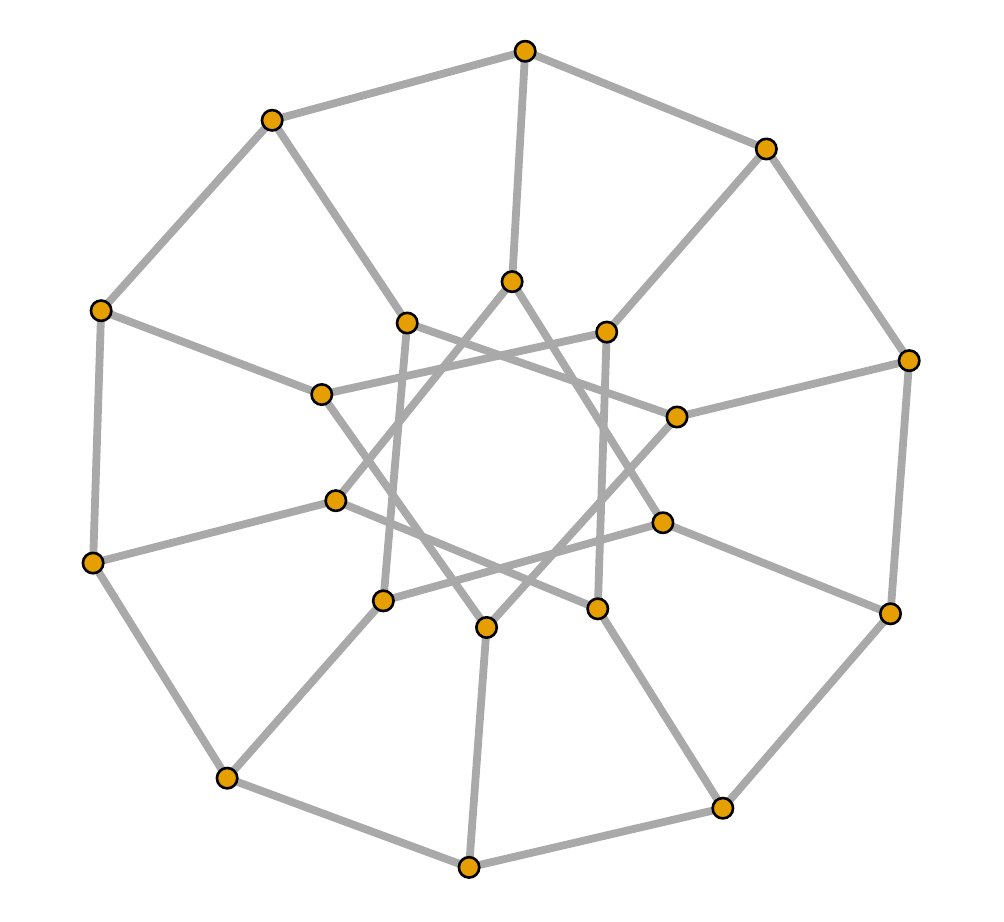}
  \caption{The 30 copies of $\mathbb{P}^1$ with 20 triple intersection points, for $2 \times 2 \times 2 \times 2 \times 2$ tensors}
\label{fig:22222}
\end{figure}

\begin{figure}[!h]
\centering
   \includegraphics[width=7cm]{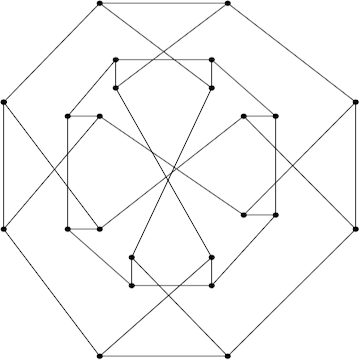}
  \caption{The 36 copies of $\mathbb{P}^1$ with 24 triple intersection points, for $4 \times 4 \times 4$ tensors}
\label{fig:444}
\end{figure}

We conclude this paper with a picture of an odeco tensor whose Type II singular vector tuples make a two-dimensional projective variety. The possible such formats are:
$$ \{ 2\times2\times2\times2\times2\times2, 2\times2\times2\times2\times3, 2\times2\times2\times4, 2\times2\times3\times3,$$
$$ 3\times3\times3\times3, 2\times2\times5, 3\times3\times5, 4\times4\times5, 5\times5\times5, 2\times3\times4, 3\times4\times4 \}.$$
We focus on those tensors of format $2 \times 2 \times 3 \times 3$. The number of components in this Type II configuration is 19. There are four copies of the projective plane $\mathbb P^2$ and 15 copies of $\mathbb P^1 \times \mathbb P^1$. The pieces are depicted in Figure \ref{fig:2233}. They are arranged around a central square, which is colored orange for ease of visibility. Each edge is a copy of $\mathbb P^1$ and its color represents the factor in which it occurs, in the order: red, yellow, blue, green. For example, green edges refer to copies of $\mathbb P^1$ of the form $\mathbb P^0 \times \mathbb P^0 \times \mathbb P^0 \times \mathbb P^1$. The configuration is not realizable in three-dimensional space: in this depiction the diagonally opposite blue and green triangles self-intersect. The generic count for the number of singular vector tuples of a tensor of this format is 98, by \cite{FO}, and the Type I singular vector tuples contribute 18 points. Therefore, this surface accounts for 80 points, which re-appear under a general perturbation of an odeco tensor of this format.
 
\begin{figure}[!h]
\centering
   \includegraphics[width=9cm]{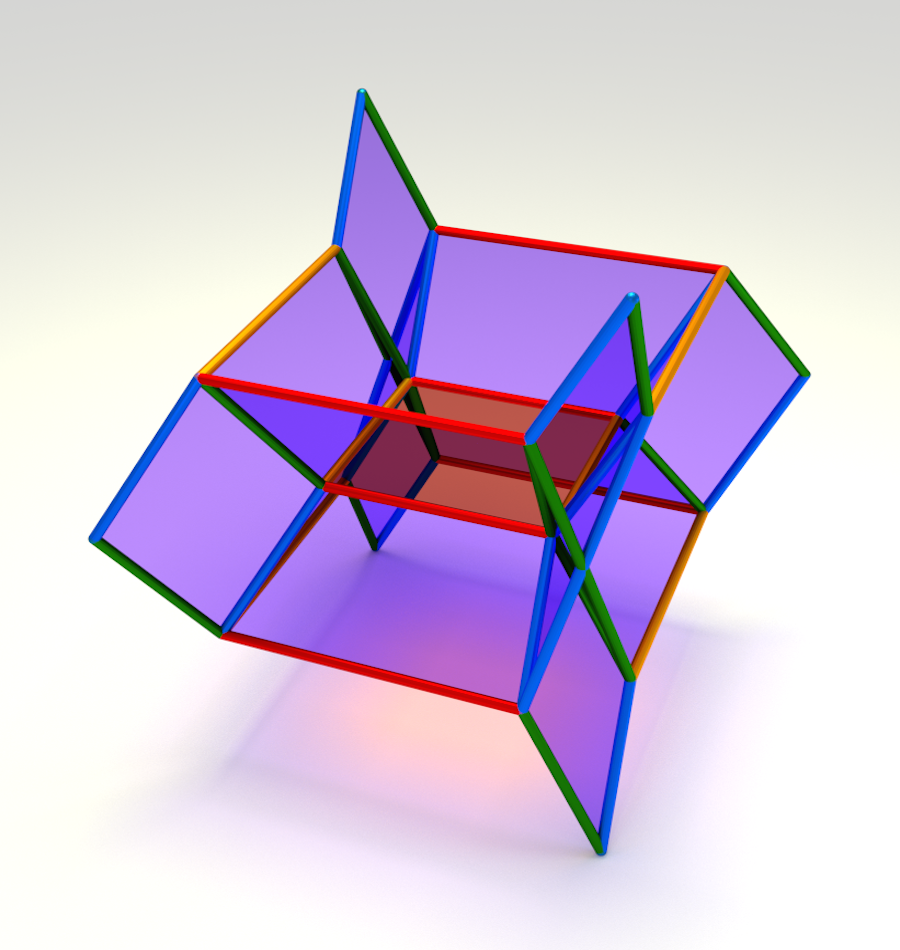}
  \caption{The arrangement of 19 two-dimensional facets, for $2 \times 2 \times 3 \times 3$ tensors}
\label{fig:2233}
\end{figure}

\subsection*{Acknowledgements}
We would like to thank our advisor Bernd Sturmfels for his support during this project. We would also like to thank Martin Helmer for reading an earlier draft of this paper. Figure \ref{fig:2233} was made by Thilo Roerig, to whom we are very grateful, using the open-source 3D graphics and animation software `Blender'.

Anna Seigal is supported by a Simons Fellowship in Mathematics, and Elina Robeva is supported by the UC Berkeley Mathematics Department.


\newcommand{\Addresses}{{
  \bigskip
  \footnotesize

  E.~Robeva, \textsc{775 Evans Hall, Department of Mathematics,
    University of California, Berkeley, Berkeley, CA94720}\par\nopagebreak
  \textit{E-mail address}, E.~Robeva: \texttt{erobeva@berkeley.edu}
  
  A.~Seigal, \textsc{1062 Evans Hall, Department of Mathematics,
    University of California, Berkeley, Berkeley, CA94720}\par\nopagebreak
  \textit{E-mail address}, A.~Seigal: \texttt{seigal@berkeley.edu}
}}

\Addresses

\end{document}